\documentclass[11pt]{article}

\usepackage{makeidx}
\makeindex
\usepackage{amsmath}
\usepackage{amsfonts}
\usepackage{amssymb}
\usepackage{amsthm}
\usepackage{MnSymbol}
\usepackage{relsize}
\usepackage{etoolbox}
\theoremstyle{plain}
\newtheorem{theorem}{Theorem}[section]

\newtheorem{lemma}[theorem]{Lemma}

\theoremstyle{definition}
\newtheorem{definition}[theorem]{Definition}

\theoremstyle{remark}
\newtheorem*{remark}{Remark}

\usepackage[margin=1in]{geometry}

\begin{document}

\title{Agmon--Kolmogorov inequalities on $\ell^2(\Bbb Z^d)$}
\author{Arman Sahovic}
\date{\today}
\maketitle

\begin{abstract}

Landau--Kolmogorov inequalities have been extensively studied on both continuous and discrete domains for an entire century. However, the research is limited to the study of functions and sequences on $\Bbb R$ and $\Bbb Z$, with no equivalent inequalities in higher-dimensional spaces. The aim of this paper is to obtain a new class of discrete Landau--Kolmogorov type inequalities of arbitrary dimension:
$$
\|\varphi\|_{\ell^\infty(\Bbb Z^d)} 
\leq
\mu_{p,d}\|\nabla_D\varphi\|^{p/2^d}_{\ell^2(\Bbb Z^d)}\,  \|\varphi\|^{1-p/2^d}_{\ell^2(\Bbb Z^d)},
$$ 
where the constant $\mu_{p,d}$ is explicitly specified. In fact, this also generalises the discrete Agmon inequality to higher dimension, which in the corresponding continuous case is not possible. 

\end{abstract}

\section{Introduction}

\noindent In 1912, G. H. Hardy, J. E. Littlewood and G. P\'{o}lya 
(see \cite{Hardy}) proved the following inequalities for a function $f\in L^2(\Bbb R)$: 
\begin{equation}\label{HLP1}
\|f'\|_{L^2(-\infty,\infty)}\leq \|f\|^{1/2}_{L^2(-\infty,\infty)}\|f''\|^{1/2}_{L^2(-\infty,\infty)},
\end{equation}
\begin{equation}\label{HLP2}
\|f'\|_{L^2(0,\infty)}\leq \sqrt{2}\|f\|^{1/2}_{L^2(0,\infty)}\|f''\|^{1/2}_{L^2(0,\infty)},
\end{equation}
with the constants $1$ and $\sqrt{2}$ being sharp. These results sparked interest in inequalities involving functions, their derivatives and integrals for a century to come. Specifically, in 1913, E. Landau (see \cite{Lan1913}) proved the following inequality:
For $\Omega\subseteq\Bbb R$, and $f\in L^\infty(\Omega)$:
$$
\| f'\|_{L^\infty(\Omega)} \leq \sqrt{2}\| f''\|_{L^\infty(\Omega)}^{1/2} \| f\|_{L^\infty(\Omega)}^{1/2},
$$
with the constant $\sqrt{2}$ being sharp. This result in turn was motivation for A. Kolmogorov (see \cite{Kol39}), where in 1939 he found sharp constants for the more general case, using a simple, but very effective inductive argument to extend the case to higher order derivatives:
$$
\| f^{(k)}\|_{L^\infty(\Omega)} \leq C(k,n)\| f^{(n)}\|_{L^\infty(\Omega)}^{k/n} \| f\|_{L^\infty(\Omega)}^{1-k/n},
$$

where, for $k, n\in\Bbb N $ with $1 \leq k < n$, he determined the best constants $C(k,n)\in\Bbb R$ for $\Omega=\Bbb R$. Since then, there has been a great deal of work on what are nowadays known as the Landau--Kolmogorov inequalities, which are in their most general form: 
$$
\| f^{(k)}\|_{L^p} \leq K(k,n,p,q,r)\,\,\| f^{(n)}\|_{L^q}^{\alpha} \| f\|_{L^r}^{\beta},
$$

with the minimal constant $K = K(k,n,p,q,r)$. The real numbers $p, q, r \geq 1$;  $k, n\in\Bbb N $ with $(0 \leq k < n)$ and $\alpha, \beta\in\Bbb R$ take on values for which the constant $K$ is finite (see \cite{Gab1967}).

However, literature on discrete equivalents of those inequalities remained very limited for a long time. In 1979, E. T. Copson (see \cite{Cop79})  was one of the first to find equivalent results for sequences, series and difference operators.
Indeed, he found the discrete equivalent to \eqref{HLP1} and \eqref{HLP2}. For a square summable sequence, $\{a(n)\}_{n\in\Bbb Z}\in\ell^2(\Bbb Z)$ and a difference operator $(Da)(n):=a(n+1)-a(n)$, we have:
\begin{equation}\label{Cop1}
\|Da\|_{\ell^2(-\infty,\infty)}\leq \|a\|^{1/2}_{\ell^2(-\infty,\infty)}\|D^2a\|^{1/2}_{\ell^2(-\infty,\infty)},
\end{equation}
\begin{equation}\label{Cop2}
\|Da\|_{\ell^2(0,\infty)}\leq \sqrt{2}\|a\|^{1/2}_{\ell^2(0,\infty)}\|D^2a\|^{1/2}_{\ell^2(0,\infty)},
\end{equation}

with the constants $1$ and $\sqrt{2}$ yet again being sharp. Z. Ditzian (see \cite{Dit2}) then extended those results to establish best constants for a variety of Banach spaces, adding equivalent results for continuous shift operators $f(x+h)-f(x);\,\, x\in \Bbb R, f\in\L^2(\Bbb R)$. 

Comparing inequalities such as \eqref{HLP1} and \eqref{HLP2}, with \eqref{Cop1}  and \eqref{Cop2} respectively, it was suspected that sharp constants were identical for equivalent discrete and continuous Landau--Kolmogorov inequalities for $1\leq p=q=r\leq\infty$. Indeed, in the cases $p=1, 2, \infty$, this was true for the whole and semi-axis. However, the general case has since been shown to be false, as for example demonstrated in \cite{KKZ88} by M. K. Kwong and A. Zettl, where they prove that for many values of $p$, the discrete constants are strictly greater than the continuous ones. 

Another important special case of the Landau--Kolmogorov inequalities is the Agmon inequality, proven by S. Agmon (see \cite{Agmon10}). Viewed as an interpolation inequality between $L^{\infty}(\Bbb R)$ and $L^{2}(\Bbb R)$, he states the following:
$$
\|f\|_{L^\infty(\Bbb R)}\leq \|f\|^{1/2}_{L^2(\Bbb R)} \|f'\|^{1/2}_{L^2(\Bbb R)}.
$$
Thus, throughout this paper we shall call, for a domain $\Omega$, a function $f\in L^2(\Omega)$, a sequence $\varphi\in\ell^2(\Omega)$, $\alpha,\,\beta$ being $\Bbb Q$-valued functions of the integers $k,\,n$ with $k\le n$ and constants $C(\Omega,k,n),$ $ \,D(\Omega,k,n)\in \Bbb R$:
\begin{equation}\label{Agmon-Kolmogorov-Continuous}
\|f^{(k)}\|_{L^\infty(\Omega)}\leq C(\Omega,k,n)\,\,\|f\|^{\alpha(k,n)}_{L^2(\Omega)} \,\|f^{(n)}\|^{\beta(k,n)}_{L^2(\Omega)},
\end{equation}
\begin{equation}\label{Agmon-Kolmogorov-Discrete}
\|D^k \varphi\|_{\ell^\infty(\Omega)}\leq D(\Omega,k,n)\,\|\varphi\|^{\alpha(k,n)}_{\ell^2(\Omega)} \,\|D^n \varphi\|^{\beta(k,n)}_{\ell^2(\Omega)},
\end{equation}
Agmon--Kolmogorov inequalities, where \eqref{Agmon-Kolmogorov-Discrete}, for $\Omega:=\Bbb Z^d$ will be the central concern of this paper. Specifically we only require the case wherew $k=0$ and $n=1$, whereas the other inequalities, i.e. those concerned with higher order, have been discussed in \cite{Sah13}.

\section{Agmon--Kolmogorov Inequalities over $\Bbb Z^d$}

We introduce our notation for the $d$-dimensional inner product space of square summable sequences.
For a vector of integers $\zeta:=(\zeta_1,\ldots,\zeta_d)\in\Bbb Z^d$, we say $\{\varphi(\zeta)\}_{\zeta\in\Bbb Z^d}\in\ell^2(\Bbb Z^d)$, if and only if the following norm is finite: 
$$
\|\varphi\|_{\ell^2({\Bbb Z^d})}:=\Big(\sum_{\zeta\in\Bbb Z^d} |\varphi(\zeta)|^p\Big)^{1/2}.
$$
Then, for $\varphi,\,\phi\in\ell^2(\mathbb{Z}^d)$, we let $<.,.>_d$ be the inner product on $\ell^2(\mathbb{Z}^d)$:
$$
\langle\varphi, \phi\rangle_d:=\sum_{\zeta\in\Bbb Z^d}
 \varphi(\zeta)\overline{\phi(\zeta)}.
$$

\noindent We then let $D_1,\ldots,D_d$ be the partial difference operators defined by:
$$
(D_i \varphi) (\zeta):= \varphi(\zeta_1,\ldots,\zeta_i+1,\ldots ,\zeta_d)-\varphi(\zeta_1,\ldots,\zeta_d),
$$
\noindent The discrete gradiant $\nabla_D$ shall thus take the following form:
$$
\nabla_D \varphi(\zeta_1,\zeta_2, \ldots, \zeta_d) = \big(D_1\varphi(\zeta), D_2\varphi(\zeta),\ldots, D_d \varphi(\zeta)\big).
$$

\noindent Thus, combining this definition with that of our norm above, we obtain:
$$
\|\nabla_D \varphi\|^2_{\ell^2({\Bbb Z^d})}=\|D_1 \varphi\|^2_{\ell^2({\Bbb Z^d})}+\ldots +\|D_d \varphi\|^2_{\ell^2({\Bbb Z^d})}.
$$
Further, we require the following notation:

\begin{definition}
For a sequence $\varphi(\zeta)\in\ell^2(\Bbb Z^d)$ with $\zeta:=(\zeta_1,...,\zeta_d)\in\Bbb Z^d$, for $0\leq k\leq d$ we define:
$$
[\varphi]_k:=\left(\sum_{\zeta_1\in \Bbb Z}...\sum_{\zeta_k\in \Bbb Z}|\varphi(\zeta)|^2\right)^{1/2}.
$$

\end{definition}

\begin{remark}
We identify that $[\varphi]_0:=|\varphi(\zeta)|$ and if we apply this operator for $k=d$, i.e. sum across all coordinates, we obtain the $\ell^2(\Bbb Z^d)$-norm:
$$
\left[\varphi\right]_d= \|\varphi\|_{\ell^2{(\Bbb Z^d)}}.
$$
\end{remark}

\noindent We are interested in a higher-dimensional version of the discrete Agmon inequality (see \cite{Sah10}), which estimates the sup-norm of a sequence $\phi\in\ell^2(\Bbb Z)$ as follows:
$$
\|\phi\|^2_{\ell^\infty(\Bbb Z)}\leq\|\phi\|_{\ell^2(\Bbb Z)}\|D\phi\|_{\ell^2(\Bbb Z)}.
$$  
Thus we commence by 'lifting' this estimate to encompass more variables:
\begin{lemma}[Agmon--Cauchy Inequality]\label{Agmon-CauchyArbDim} For the operator $D_{k+1}$, acting on a sequence $\varphi(\zeta)\in\ell^2(\Bbb Z^d)$, we have:
\begin{equation*}
\sup_{\zeta_{k+1}\in\Bbb Z}\left[\varphi\right]_k \le 
 \left[D_{k+1}\varphi\right]_{k+1}^{1/2} \,  \left[\varphi\right]_{k+1}^{1/2}.
\end{equation*}

\end{lemma}

\begin{proof}

\noindent Using the discrete Agmon inequality on the $(k+1)^{th}$ coordinate, we find:
\begin{equation*}
 | \varphi(\zeta_1,\ldots,\zeta_d)|^2
\le 
\Big(\sum_{l\in \Bbb Z} | D_{k+1}\varphi(\zeta_1,\ldots,\zeta_k,l,\zeta_{k+2},\ldots,\zeta_{d})|^2\Big)^{1/2} 
\Big( \sum_{l\in \Bbb Z} | \varphi(\zeta_1,\ldots,\zeta_k,l,\zeta_{k+2},\ldots,\zeta_d)|^2\Big)^{1/2}.
\end{equation*}

\noindent Now we sum with respect to the other coordinates:
$$
 \sum_{\zeta_1\in \Bbb Z}...\sum_{\zeta_k\in \Bbb Z}| \varphi(\zeta_1,\ldots,\zeta_d)|^2  \le \qquad\qquad\qquad\qquad\qquad
\qquad\qquad\qquad\qquad\qquad\qquad\qquad\qquad\qquad\qquad
 $$
 $$
\sum_{\zeta_1\in \Bbb Z}...\sum_{\zeta_k\in \Bbb Z}\left[\Big(\sum_{l\in \Bbb Z} | D_{k+1}\varphi(\zeta_1,\ldots,\zeta_k,l,\zeta_{k+2},\ldots,\zeta_{d})|^2\Big)^{1/2}  \,
\Big( \sum_{l\in \Bbb Z} | \varphi(\zeta_1,\ldots,\zeta_k,l,\zeta_{k+2},\ldots,\zeta_d)|^2\Big)^{1/2}\right],
$$ 

\noindent and use the Cauchy--Schwartz inequality on the $k^{th}$ coordinate:
\begin{eqnarray*}
 \sum_{\zeta_1\in \Bbb Z}...\sum_{\zeta_k\in \Bbb Z}| \varphi(\zeta_1,\ldots,\zeta_d)|^2 
 &\le &
\sum_{\zeta_1\in \Bbb Z}...\sum_{\zeta_{k-1}\in \Bbb Z}\Big[\Big(\sum_{\zeta_k\in \Bbb Z}\sum_{l\in \Bbb Z} | D_{k+1}\varphi(\zeta_1,\ldots,\zeta_k,l,\zeta_{k+2},\ldots,\zeta_{d},)|^2\Big)^{1/2}  \cdot\\
&&
\,\Big( \sum_{\zeta_k\in \Bbb Z}\sum_{l\in \Bbb Z} | \varphi(\zeta_1,\ldots,\zeta_k,l,\zeta_{k+2},\ldots,\zeta_d)|^2\Big)^{1/2}\Big].
\end{eqnarray*}

\noindent We repeat this process to finally obtain:
\begin{eqnarray*}
\sum_{\zeta_1\in \Bbb Z}...\sum_{\zeta_k\in \Bbb Z}| \varphi(\zeta_1,\ldots,\zeta_d)|^2 
 &\le &
\Big(\sum_{\zeta_1\in \Bbb Z}...\sum_{\zeta_k\in \Bbb Z}\sum_{l\in \Bbb Z} | D_{k+1} \varphi(\zeta_1,\ldots,\zeta_k,l,\zeta_{k+2},\ldots,\zeta_{d})|^2\Big)^{1/2}  \cdot \\
&&
\,\Big( \sum_{\zeta_1\in \Bbb Z}...\sum_{\zeta_k\in \Bbb Z}\sum_{l\in \Bbb Z} | \varphi(\zeta_1,\ldots,\zeta_k,l,\zeta_{k+2},\ldots,\zeta_d)|^2\Big)^{1/2}.
\end{eqnarray*}

\end{proof}
\noindent We estimate the $\ell^2{(\Bbb Z^d)}$-norm of a partial difference operator with the $\ell^2{(\Bbb Z^d)}$-norm of the sequence itself:
\newpage
\begin{lemma}\label{DownDifferenceN}
For a sequence $\varphi\in\ell^2(\Bbb Z^d)$ and for $i\in\{1,\ldots,d\}$, we have:
$$
 \,\|D_i \varphi\|_{\ell^2{(\Bbb Z^d)}}
 \leq
2 \|\varphi\|_{\ell^2{(\Bbb Z^d)}}.
$$
\end{lemma}

\begin{proof} We show the argument for $D_1$ and note that due to symmetry the other cases follow immediately.
\begin{eqnarray*}
 \|D_1 \varphi\|^2_{\ell^2{(\Bbb Z^d)}}
 &=&
 \sum_{\zeta\in\Bbb Z^d} | \varphi(\zeta_1+1,\ldots,\zeta_d) - (\zeta_1,\ldots,\zeta_d)|^2\\
 &\le & 
  2\,  \Big( \sum_{\zeta\in\Bbb Z^d} |\varphi(\zeta_1+1,\ldots,\zeta_d)|^2 +  \sum_{\zeta\in\Bbb Z^d} |\varphi(\zeta_1,\ldots,\zeta_d)|^2\Big)\\ 
&=& 
4\, \sum_{\zeta\in\Bbb Z^d}| \varphi(\zeta_1,\ldots,\zeta_d)|^2\\
&=&
4 \|\varphi\|^2_{\ell^2{(\Bbb Z^d)}}.
\end{eqnarray*}
\end{proof}
\noindent This implies that we can obtain an estimate for any mixed difference operator as follows:
$$
\|D_1\ldots D_k \varphi\|_{\ell^2{(\Bbb Z^d)}}\le 2\|D_1\ldots D_{l-1}D_{l+1}\ldots D_k \varphi\|_{\ell^2{(\Bbb Z^d)}}.
$$
Therefore, by eliminating $l$ difference operators, our inequality will contain the constant $2^l$.

\noindent We arrive at our main result, the Agmon--Kolmogorov inequalities on $\ell^2(\Bbb Z^d)$. 

\begin{theorem} \label{AgmonArbDim}

For a sequence $\varphi\in\ell^2(\Bbb Z^d)$, and $p\in\{1,\ldots,2^{d-1}\}$:
$$
\|\varphi\|_{\ell^\infty(\Bbb Z^d)} 
\leq
\mu_{p,d}\|\nabla_D\varphi\|^{p/2^d}_{\ell^2(\Bbb Z^d)}\,  \|\varphi\|^{1-p/2^d}_{\ell^2(\Bbb Z^d)},
$$

\noindent where
$$
\mu_{p,d}:=\left(\frac{\kappa_{p,d}}{{d^{p/2}}}\right)^{1/2^d},
$$
and $\kappa_{p,d}$ is a constant to be determined in the following section.

\end{theorem}

\begin{proof}$\ $\\
\noindent We use Lemma \ref{Agmon-CauchyArbDim} and Lemma \ref{DownDifferenceN} repeatedly:
\begin{eqnarray*}
\|\varphi\|_{\ell^\infty(\Bbb Z^d)}
&\leq &
\left[D_{1} \varphi\right]^{1/2}_{1} \,  \left[\varphi\right]^{1/2}_{1}\\
&\leq &
\left[D_{2}D_1\varphi\right]^{1/4}_{2} \,  \left[D_1\varphi\right]^{1/4}_{2}\left[D_{2}\varphi\right]^{1/4}_{2} \,  \left[\varphi\right]^{1/4}_{2}\\
%
%
&\vdots &\\
&\leq & 
\left[D_d\ldots D_{1}\varphi\right]^{1/2^d}_{d}\,\ldots\,\,\ldots \,  \left[\varphi\right]^{1/2^d}_{d}\\
& = &
\|D_d\ldots D_{1}\varphi\|^{1/2^d}_{\ell^2(\Bbb Z^d)}\,\ldots \,\,\ldots \,  \|\varphi\|^{1/2^d}_{\ell^2(\Bbb Z^d)}\\
\Rightarrow \|\varphi\|_{\ell^\infty(\Bbb Z^d)}^{2^d}& \leq &
\|D_d\ldots D_{1}\varphi\|_{\ell^2(\Bbb Z^d)}\,\ldots\,\,\ldots  \,  \|\varphi\|_{\ell^2(\Bbb Z^d)}.
\end{eqnarray*}

\noindent We have generated an estimate by $2^d$ norms, with exactly $2^{d-1}$ norms  originating from the term $\left[D_{1}\varphi\right]^{1/2}_{1}$. All those will thus involve the operator $D_1$, or more formally: $|\Xi_1|=2^{d-1}$, where we let 
$$
\Xi_1:=\left\{\|D_{a_1}\ldots D_{a_k} D_1\varphi\|_{\ell^2(\Bbb Z^d)}\,\big|\,\, a_i\neq a_j \,\forall \,i\neq j\, ;\,\{a_1,\ldots,a_k\}\subset\{2,\ldots,d\}\right\}.
$$
\noindent We note that we could also employ estimates by $\|D_i\varphi\|_{\ell^2(\Bbb Z^d)}$ for any $i\in\{1,\ldots,2^d\}$, but our inequality will not change due to our symmetrising argument.
Similarly, we have $2^{d-1}$ norms originating from the term $\left[ \varphi \right]^{1/2}_{1}$, whose estimates will not involve the operator $D_1$. Hence $|\Xi_2|=2^{d-1}$, where we let 
$$
\Xi_2:=\left\{\|D_{a_1}\ldots D_{a_k} \varphi\|_{\ell^2{(\Bbb Z^d)}}\,\big|\,\, a_i\neq a_j \,\forall \,i\neq j\, ;\,\{a_1,\ldots,a_k\}\subset\{2,\ldots,d\}\right\}.
$$

\noindent We will now apply Lemma \ref{DownDifferenceN} repeatedly, to reduce the order of the operator inside the norms to either 0 or 1. We recognise that we have to estimate all $^1\xi\in\Xi_1$ by $^1\xi_1:=\|D_1\varphi\|_{\ell^2(\Bbb Z^d)}$ or alternatively by $\|\varphi\|_{\ell^2(\Bbb Z^d)}$. \\
Hence, we choose a $p\in\{0,\ldots, 2^{d-1}\}$ 
to estimate $p$ elements in $\Xi_1$ by $\|D_1\varphi\|_{\ell^2(\Bbb Z^d)}$,  leaving $2^{d-1}-p$ elements in $\Xi_1$ to be estimated by $\|\varphi\|_{\ell^2(\Bbb Z^d)}$.
\noindent However, for all $2^{d-1}$ elements $^2\xi\in\Xi_2$, we have to provide an estimate by $^2\xi_1:=\|\varphi\|_{\ell^2(\Bbb Z^d)}$ only. This means we have $2^d-p$ elements in $\Xi:=\Xi_1\bigcup\Xi_2$ to be estimated by $\|\varphi\|_{\ell^2(\Bbb Z^d)}$: 
$$
\|\varphi\|_{\ell^\infty(\Bbb Z^d)} ^{2^d} \leq
\kappa_{p,d}\|D_{1}\varphi\|^{p}_{\ell^2(\Bbb Z^d)} \,  \|\varphi\|^{2^d-p}_{\ell^2(\Bbb Z^d)}.
$$
where $\kappa_{p,d}$ remains a constant of the form $2^z$ with $z\in\Bbb Q$, which we leave to be identified in the next section. We thus obtain the following estimate:
$$
\|\varphi\|^{2^{d+1}/p}_{\ell^\infty(\Bbb Z^d)}  \leq
\kappa_{p,d}^{2/p}\|D_{1}\varphi\|^{2}_{\ell^2(\Bbb Z^d)} \,  \|\varphi\|^{(2^{d+1}-2p)/p}_{\ell^2(\Bbb Z^d)}.
$$
We now exploit the symmetry of the argument:
$$
d\,\|\varphi\|^{2^{d+1}/p}_{\ell^\infty(\Bbb Z^d)} 
\leq
\kappa_{p,d}^{2/p}\left(\|D_{1}\varphi\|^{2}_{\ell^2(\Bbb Z^d)}+\ldots+\|D_{d}\varphi\|^{2}_{\ell^2(\Bbb Z^d)} \right)\,  \|\varphi\|^{(2^{d+1}-2p)/p}_{\ell^2(\Bbb Z^d)}
$$
$$
=
\kappa_{p,d}^{2/p}\|\nabla_D\varphi\|^{2}_{\ell^2(\Bbb Z^d)}\,  \|\varphi\|^{(2^{d+1}-2p)/p}_{\ell^2(\Bbb Z^d)} ,
$$
and finally rearrange:
$$
\|\varphi\|_{\ell^\infty(\Bbb Z^d)} 
\leq
\left(\frac{\kappa_{p,d}}{{d^{p/2}}}\right)^{1/2^d}\|\nabla_D\varphi\|^{p/2^d}_{\ell^2(\Bbb Z^d)}\,  \|\varphi\|^{1-p/2^d}_{\ell^2(\Bbb Z^d)}.
$$
\end{proof}
\newpage 

\section{The Constant $\kappa_{p,d}$}

\noindent It remains to identify the constant $\kappa_{p,d}$, we thus give:

\begin{theorem}\label{Agmond-constant}

We have, for arbitrary dimension $d$ and $p\in\{1,\ldots,2^{d-1}\}$:
$$
\kappa_{p,d}=
2^{\,d\,\cdot\,2^{d-1}-p}.
$$

\end{theorem}

\noindent We will break the proof down into several steps. 
The method for finding $\kappa_{p,d}$ will rely largely on the following observation:

Let $\tau (\xi)$ be the order of the operator contained in any given $\xi\in\Xi$. Then we let $\Omega_i:=\{\xi\,\big|\,\tau (\xi)=i\}$, be the set of all terms in the estimate whose operator has a given order $i$. In $\Xi_1$  we have $1\leq i \leq d$, and in $\Xi_2$, $0\leq i \leq d-1$.

\begin{lemma}\label{PascalOmega}
\noindent For the size of $\Omega_i$, we have for $d\geq 2$:

\noindent For $\Xi_1$: 
$$
|\Omega_i|=\binom{d-1}{i-1}, \qquad 1\leq i \leq d,
$$
and $\Xi_2$:
$$
|\Omega_i|=\binom{d-1}{i}, \qquad 0\leq i \leq d-1.
$$
\end{lemma}
\begin{proof}
We follow by induction and prove the case of $\Xi_2$, noting that the argument for $\Xi_1$ is symmetrically identical. We have already seen that the formula is correct for $d=2$ , and now we assume it is true for $d=l$, i.e. for $0\leq i \leq l-1$:
$$
|\Omega_i|=\binom{l-1}{i},
$$
and thus we have the following list:
\bigskip
\begin{center}
\begin{tabular}{rccccccccccccc}
$\Xi_2$& $^2\xi_{2^{d-1}}$ & $\ldots$ & $\ldots$ &$^2\xi_2$& $^2\xi_1$& $|\Omega_0|$ & $|\Omega_1|$ & $|\Omega_2|$& \ldots & $|\Omega_{l-1}|$\\
$\Bbb Z^l$: & $D_l\ldots D_2$ & $\ldots$ & $\ldots$ & $D_2$ & $1$ & $\binom{l-1}{0}$ & $\binom{l-1}{1}$ & $\binom{l-1}{2}$ & \ldots & $\binom{l-1}{l-1}$ &  \\ 
\noalign{\smallskip\smallskip}
\end{tabular}
\end{center}

\bigskip

\noindent Now each term of a given order $\tau$ will, by the Agmon--Cauchy inequality (Lemma \ref{Agmon-CauchyArbDim}), generate a term of order $\tau$ and one of order $\tau+1$. Thus we have:
\medskip
\begin{center}
\begin{tabular}{rccccccccccccc}
$\Xi_2$&$^2\xi_{2^{d}}$&$\ldots$ & $\ldots$ & $^2\xi_2$ & $^2\xi_1$& $|\Omega_0|$ & $|\Omega_1|$ & $|\Omega_2|$& \ldots & $|\Omega_{l}|$\\
$\Bbb Z^{l+1}$: & $D_{l+1}\ldots D_2$ & $\ldots$ & $\ldots$ & $D_2$ & $1$ & $\binom{l-1}{0}$ & $\binom{l-1}{0}+\binom{l-1}{1}$ & $\binom{l-1}{1}+\binom{l-1}{2}$ & \ldots & $\binom{l-1}{l-1}$ &  \\ 
\noalign{\smallskip\smallskip}
\end{tabular}
\end{center}

\bigskip

\noindent Now we apply the standard combinatorial identity $^aC_b+ {^aC_{b+1}}= {^{a+1}C_{b+1}}$ and consider ${^aC_{0}}={^aC_{a}}=1$, which immediately implies:

\bigskip
\begin{center}
\begin{tabular}{rccccccccccccc}
$\Xi_2$ & $^2\xi_{2^{d}}$ & $\ldots$ & $\ldots$ & $^2\xi_{2}$ & $^2\xi_{1}$ & $|\Omega_0|$ & $|\Omega_1|$ & $|\Omega_2|$& ...& $|\Omega_{l+1}|$\\
$\Bbb Z^{l+1}$: & $D_{l+1}\ldots D_2$ & $\ldots$ & $\ldots$ & $D_2$ & $1$ & $\binom{l}{0}$ & $\binom{l}{1}$ & $\binom{l}{2}$ & \ldots & $\binom{l}{l}$ &  \\ 
\noalign{\smallskip\smallskip}
\end{tabular}
\end{center}

\bigskip

\noindent and hence for $d=l+1$, we have:
$$
|\Omega_i|=\binom{l}{i},
$$
completing our inductive step. 

\end{proof}

\noindent As discussed previously, if we consider to estimate a given $\xi\in\Xi$ using Lemma \ref{DownDifferenceN}, we will, for example, obtain 
$
\|D_1\ldots D_k \varphi\|_{\ell^2{(\Bbb Z^d)}}\le 2\|D_1\ldots D_{l-1}D_{l+1}\ldots D_k \varphi\|_{\ell^2{(\Bbb Z^d)}}.
$ We can see that we generate a factor of 2 for every partial difference operator we eliminate, and thus have, for $^1\xi\in\Xi_1$ and $^2\xi\in\Xi_2$ with order $\tau(^1\xi)$ and $\tau(^2\xi)$ respectively:
$$
^1\xi\leq\,\, 2^{\tau(^1\xi)-1}\,\|D_1\varphi\|_{\ell^2(\Bbb Z^d)},\qquad \text{and} \qquad ^2\xi\leq\,\, 2^{\tau(^2\xi)}\,\|\varphi\|_{\ell^2(\Bbb Z^d)}.
$$

\noindent We note here that $\kappa_{p,d}$ will not depend on which $\ell^2(\Bbb Z^d)$-norms in $\Xi_1$ are chosen to be estimated by $^2\xi_1:=\,\|\varphi\|_{\ell^2(\Bbb Z^d)}$. The reason for this is transparent when considering that the sum of all the orders $\sum_{i=1}^{2^{d-1}} \tau({^1\xi_i})$ is a constant and needs to be reduced to the constant $p\cdot \tau({^1\xi_1})=p$, generating a unique $\kappa_{p,d}$.

\begin{lemma}
The $\min_p\kappa_{p,d}$ will be attained at $p=2^{d-1}$ and takes on the following explicit form:
%
$$
\kappa_{2^{d-1},d}=\mathlarger{\prod_{i=0}^{d-1}}\,\,
2^{2i\binom{d-1}{i}}.
$$

\end{lemma}

\begin{proof}
Our minimum constant for $\Xi_1$ in fact occurs if we choose all $^1\xi_1\in\Xi_1$ to be estimated by $\|D_1\varphi\|_{\ell^2{(\Bbb Z^d)}}$, i.e. choose $p=2^{d-1}$, the maximum $p$ possible.
Our minimum constant, denoted by $\rho^1_{d}$, for all terms in $\Xi_1$ will thus be:
$$
\rho^1_{d}=\mathlarger{\prod_{k=1}^{2^{d-1}}}\,\,
2^{\tau(^1\xi_k)-1}.
$$

\noindent Instead of examining each individual element $^1\xi$, we consider that all $^1\xi$ of equal order $i$ generate the same constant, namely $2^{i-1}$. Thus we collect all $^1\xi$ of the same order, and obtain: 
$$
\rho^1_{d}=\mathlarger{\prod_{i=1}^{d}}\,\,
2^{(i-1)|\Omega_i|} 
=\mathlarger{\prod_{i=1}^{d}}\,\,
2^{(i-1)\binom{d-1}{i-1}}.
$$

\noindent Then we need to estimate all $^2\xi\in\Xi_2$, and
we proceed as for $\Xi_1$. All $^2\xi$ need to be estimated by $\|\varphi\|_{\ell^2(\Bbb Z^d)}$, each generating the constant $2^i$, forming the equivalent pattern as that of $\Xi_1$. We thus obtain, for the minimal constant $\rho^2_{d}$:
$$
\rho^2_{d}=\mathlarger{\prod_{i=0}^{d-1}}\,\,
2^{i|\Omega_{i}|}
=\mathlarger{\prod_{i=0}^{d-1}}\,\,
2^{i\binom{d-1}{i}}.
$$
\vspace{0.2cm}
\noindent We now see that $\rho^2_{d}=\rho^1_{d}$, and:
$$
\kappa_{2^{d-1},d}=\rho^2_{d}\rho^1_{d}=\mathlarger{\prod_{i=0}^{d-1}}\,\,
2^{2i\binom{d-1}{i}}.
$$

\end{proof}

\noindent We are now finally in a position to prove Theorem \ref{Agmond-constant}:

\begin{proof} [Proof of Theorem \ref{Agmond-constant}]

\noindent We are left to analyse the constant's dependence on our choice of $p$. First we note that in addition to the constant generated above, we will have chosen $2^{d-1}-p$ terms to be further reduced to $\|\varphi\|_{\ell^2(\Bbb Z^d)}$, each generating a power of $2$.
Hence we additionally need to multiply $\kappa_{2^{d-1},d}$ by $2^{2^{d-1}-p}$. Thus our final constant will be:

$$
\kappa_{p,d}= 2^{2^{d-1}-p}\cdot\mathlarger{\prod_{i=0}^{d-1}}\,\,
2^{2i\binom{d-1}{i}}
=
2^{2^{d-1}-p+2\sum_{i=0}^{d-1}i\binom{d-1}{i}},
$$

\noindent Then we can simplify this further by considering the binomial formula $(1+X)^n=\sum_{k=0}^n \binom{n}{k}X^k$.
We differentiate with respect to $X$ and set $X=1$:
$$
n\cdot\,2^{n-1}=\sum_{k=0}^n k\,\binom{n}{k}.
$$
Thus we arrive at:
$$
\kappa_{p,d}=2^{d\cdot\,2^{d-1}-p}.
$$

\end{proof}

\bibliography{bibliography}{}
\bibliographystyle{alpha}

\end{document}